\documentclass[reqno]{amsart}

\usepackage{amsmath,amssymb,amsthm,amsfonts,enumerate, enumitem,hyperref,cleveref}
\usepackage{dsfont}
\usepackage[all]{xy}
\usepackage[T1]{fontenc}
\usepackage{tikz}
\usepackage{pgfplots}
\pgfplotsset{compat=1.15}
\usepackage{mathrsfs}
\usetikzlibrary{arrows}

\theoremstyle{plain}

\newtheorem{thrm}{Theorem}[section]
\newtheorem{cor}[thrm]{Corollary}

\newtheorem{lem}[thrm]{Lemma}

\theoremstyle{definition}

\newtheorem{rem}[thrm]{Remark}
\newtheorem{exm}[thrm]{Example}

\crefrangeformat{equation}{#3(#1)#4--#5(#2)#6}

\crefname{thrm}{Theorem}{Theorems}
\crefname{theorem}{Theorem}{Theorems}
\crefname{lem}{Lemma}{Lemmas}
\crefname{cor}{Corollary}{Corollaries}
\crefname{prop}{Proposition}{Propositions}
\crefname{defn}{Definition}{Definitions}
\crefname{exm}{Example}{Examples}
\crefname{rem}{Remark}{Remarks}
\crefname{conj}{Conjecture}{Conjectures}
\crefname{quest}{Question}{Questions}
\crefname{section}{Section}{Sections}
\crefname{equation}{\unskip}{\unskip}
\crefname{enumi}{\unskip}{\unskip}
\crefname{subsection}{Subsection}{Subsections}

\newcommand{\lb}{\lambda}

\newcommand{\vf}{\varphi}

\newcommand{\dl}{\delta}

\newcommand{\CC}{\mathbb{C}}

\newcommand{\sst}{\subseteq}

\newcommand{\impl}{\Rightarrow}

\begin{document}
	\title[Bilinear maps  having Jordan product property]{Bilinear maps having Jordan product property}	
	
	\author{Jorge J. Garc{\' e}s}
	\address{Departamento de Matem{\' a}tica Aplicada a la Ingenier{\' i}a Industrial, ETSIDI, Universidad Polit{\' e}cnica de Madrid, Madrid, Spain}
	\email{j.garces@upm.es}
	
	\author{Mykola Khrypchenko}
	\address{Departamento de Matem\'atica, Universidade Federal de Santa Catarina, Campus Trindade, Florian\'opolis, SC, CEP: 88040--900, Brazil}
	\email{nskhripchenko@gmail.com}

	
	\begin{abstract}
		We study symmetric continuous bilinear maps $V$ on a C$^*$ -algebra $A$ that have the Jordan product property at a fixed element $z\in A$. We show that, whenever $A$ is a finite direct sum of infinite simple von Neumann algebras, such a map $V$ has the square-zero property. Then, it is proved that $V(a,b)=T(a\circ b)$ for some bounded linear map $T$ on $A$. As a consequence, Jordan homomorphisms and derivations at $z\in A$ are characterized.
	\end{abstract}
	
	\maketitle
	
	\tableofcontents
	
	\section*{Introduction}
	Let $A,B$ be algebras and $\vf:A\to B$ an additive map. We say that $\vf$ \textit{preserves square-zero elements} or \textit{is a  square-zero preserver} if $\vf$ satisfies the property
	\begin{align}\label{vf sqaure-zero}
		a^2=0\Rightarrow \vf(a)^2=0.  
	\end{align}
	If in \cref{vf sqaure-zero} we have ``$\Leftrightarrow$'' instead of ``$\Rightarrow$'', then we say that $\vf$ \textit{strongly preserves square-zero elements}.
	In \cite{Semrl-square-zero}, {\v{S}}emrl  showed that a unital additive surjective map $\vf:B(H)\to B(H)$ that strongly preserves square-zero elements is either an automorphism or an antiautomorphism. The author leaves as an open problem to obtain a non-unital version of his result. In \cite{Bai_Hou_Annih_Polyn},  Bai and Bou gave an affirmative answer to {\v{S}}emrl's problem by showing that a strong square-zero preserving bijection $\vf:B(H)\to B(K)$ is a multiple of a Jordan isomorphism. Moreover, they also showed that if $\vf$ is continuous, then strongness can be dropped while bijectivity can be relaxed to surjectivity (see \cite[Theorem 2.3]{Bai_Hou_Annih_Polyn}). Derivation-like mappings can also be studied in terms of square-zero elements. A linear map $\delta:A\to X$ (where $A$ is an algebra and $X$ an $A$-bimodule) is said to be a \textit{square-zero derivation} if 
	$$
	a^2=0\impl a\delta(a)+\delta(a)a=0
	$$ 
	as defined in \cite{Sq-Zer-det}. 
	
	Paralleling the theory of zero product determined algebras, square-zero determined algebras were introduced in \cite{Sq-Zer-det} by Ma, Ding and Wang.  An algebra $A$ is said to be \textit{square-zero determined} if for every vector space $X,$ every bilinear map $V:A\times A\to X$ that has the property  
	\begin{align}\label{def bilinear sq-zero}
		a^2=0 \Rightarrow V(a,a)=0 
	\end{align} 
	(called \textit{square-zero property}) is of the form $V(a,b)=T(a\circ b)$ for some linear map $T:A\to X.$ Here and below by $a\circ b$ we mean $ab+ba$. In the aforementioned contribution, the authors show that the algebra of all $n\times n$ strictly upper triangular matrices $N_n(R)$ over a commutative unital ring $R$ such that $\frac{1}{2}\in R$ is square-zero determined. Bilinear maps that have the square-zero property provide a unified approach to both square-zero preserving linear maps and square-zero derivations.  Although square-zero preservers have been considered by many authors (see, for instance, \cite{Semrl-square-zero,zbMATH02196733,zbMATH05682732}), to our knowledge the only known examples of square-zero determined algebras are those considered in \cite{Sq-Zer-det}. In \cref{thm BH square-zero determined} we show that certain von Neumann algebras (in particular, $B(H)$ and sums of $B(H)$'s with $H$ infinite-dimensional Hilbert spaces) are square-zero determined.
	
	After studying square-zero determined algebras, we introduce algebras determined by Jordan products at a fixed point. This concept is the Jordan version of algebras determined by products at a fixed element introduced in \cite{Gar_Myk_bil-comp} and is inspired by the theory of zero-product determined algebras (see \cite{Bresar-zpd}). As in the case of zero-product determined algebras, this theory is introduced in order to deal, in an unified manner, with Jordan homomorphisms and derivations at a fixed element. Recall that a linear map $\delta :A\to X$ (where $A$ is an algebra and $X$ is an $A$-bimodule) is said to be a \textit{derivation} (or \textit{derivable}) at some $z\in A$ if 
	\begin{align}\label{a-circ-b=a=>dl(a-circ-b)=a-circ-dl(b)+b-circ-dl(a)}
		a\cdot b=z\impl\delta(a\cdot b)=a\cdot \delta(b)+\delta(a)\cdot b.    
	\end{align}
	Replacing $\cdot$ by $\circ$ in \cref{a-circ-b=a=>dl(a-circ-b)=a-circ-dl(b)+b-circ-dl(a)}, one gets a \textit{Jordan derivation} (or a \textit{Jordan derivable map}) at $z$. This concept probably appeared (with $a^2$ instead of $a\circ b$) for the first time in \cite{Wu-all-deriv-0-1}, where Wu proved that every Jordan derivation at $0$ or $1$ on $B(H)$ is a derivation. In the more general setting of von Neumann algebras, the authors study in \cite{Dol-Kuz-Qi-jorda-der} linear maps on $B(H)$ that are Jordan derivable at an idempotent with infinite rank and co-rank. In the case of linear maps on a unital Banach algebra with values in a unital Banach $A$-bimodule, it is known that every Jordan derivable map at the unit element is a Jordan derivation (see \cite[Corollary 2.2]{jordn-der-inv-banach-alg}). Until very recently, most results on derivable maps deal with special elements like $0,1$ or an idempotent or very specific algebras like $B(H)$, triangular algebras or nest algebras. In \cite{Chen_jordan_der} Chen, Guo, and Qin showed that every bounded linear map on $B(H)$ that is Jordan derivable at a nonzero operator is a derivation. To our knowledge, the case of Jordan derivable maps at zero is the one where more general results can be found. Indeed, every linear map on a C$^*$-algebra $A$ with values in an essential Banach $A$-bimodule $X$ that is Jordan derivable at zero, is of the form $d+G$ where $d$ is a derivation and $G$ is a bimodule homomorphism (see \cite[Theorem 4.1 and Remark 4.2]{Al-Bres-Ex-VIll-jordanzero} or \cite[Theorem 8.6]{Bresar-zpd}).
	
	Now we move our interest to generalizing the concept of a Jordan homomorphism. Hence we say that a linear map $\varphi:A\to B$ between two algebras is a \textit{Jordan homomorphism at $z\in A$} if for all $a,b\in A$: 
	$$
	a\circ b=z\impl\vf(a)\circ \vf(b)=\vf(z).
	$$ 
	If $z=0$, then we say that $\vf$ \textit{preserves zero Jordan products}. It follows from the result by Catalano, Hsu and Kapalko~\cite{Catalano-Hsu-Kapalko19} that a Jordan homomorphism $\vf:M_n(\CC)\to M_n(\CC)$ at a fixed element of $M_n(\CC)$ preserves square-zero matrices. If, moreover, $\vf$ is bijective, then \cite[Corollary 2]{Semrl93-square-zero} implies that $\vf$ is a nonzero multiple of an automorphism or an antiautomorphism of $M_n(\CC)$.
	
	In this paper, we study a partial generalization of {\v{S}}emrl's problem. We consider symmetric bilinear maps $V:A\times A\to X$
	having square-zero property, i.e. satisfying \cref{def bilinear sq-zero}. An algebra $A$ is said to be \textit{symmetrically square-zero determined} if for any symmetric bilinear map $V:A\times A\to X$ that has square-zero property there exists a linear map $T:A\to X$ such that $V(a,b)=T(a\circ b)$.
	
	In \cref{sec-prelim}, we give an example that motivates the choice of algebras under consideration and proceed with the definitions that deal with the Jordan product.
	
	In \cref{sec-sq-zero}, we study symmetric bilinear maps that have the square-zero property on infinite simple von Neumann algebras. We show in \cref{thm BH square-zero determined} that such maps preserve zero products of symmetric elements. As a consequence, infinite simple von Neumann algebras, as well as their finite direct sums and $c_0$-sums, are symmetrically square-zero determined Banach algebras.
	
	\cref{sec-Jordan-property} is devoted to symmetric bilinear maps having the Jordan product property at a fixed element. It turns out by \cref{lem transfer pp normal to sq-zero} that such maps on Rickart C$^*$-algebras have the square-zero property. So, using the results of \cref{sec-sq-zero}, we deduce in \cref{cor sums A simple determined any} that finite direct sums of simple infinite von Neumann algebras are symmetrically determined by Jordan products at any fixed element. 
	
	In \cref{sec-appl} we apply the results of \cref{sec-sq-zero,sec-Jordan-property} to suitable symmetric bilinear maps in order to get descriptions of square-zero preservers and Jordan homomorphisms at a fixed point on a finite direct sum of infinite von Neumann algebras (\cref{thrm vf pp at square zero}). It turns out that the results on square-zero preservers can be extended to $c_0$-sums of infinite simple von Neumann algebras (\cref{Jordan-homo-at-point-c_0-sum}). Square-zero derivations and Jordan derivable maps at a fixed point  are also treated in the case of finite sums of infinite simple von Neumann algebras in \cref{thrm derivable maps,cor derivable maps X=A}.

	\section{Preliminaries}\label{sec-prelim}
	
	In what follows, we shall restrict ourselves to certain classes of C$^*$-algebras. Observe that, by Gelfand theory, an abelian C$^*$-algebra has no nonzero square-zero elements. Thus, we shall focus on noncommutative C$^*$-algebras. In \cite{Sq-Zer-det} it is shown that $M_3(\CC)$ is not symmetrically square-zero determined. Actually, this example can be easily adapted to show that no noncommutative finite-dimensional C$^*$-algebra is symmetrically square-zero determined.
	\begin{exm}\label{nonexample}
		Let $A$ be a noncommutative finite-dimensional C$^*$-algebra. There exist $n_1, \dots,n_k\in \mathbb{N}$ such that $A\cong M_{n_1}(\CC)\oplus \dots \oplus M_{n_k}(\CC).$ Since $A$ is not commutative, we can assume, without loss of generality,  that $n_1>1.$ Let $\pi_i:A\to M_{n_i}$ be the canonical projections. Let us define $V:A\times A\to \CC$ by 
		$$
		V(a,b)=tr(\pi_1(a))tr(\pi_1(b)).
		$$ 
		Suppose that there exists a linear map $T:A\to \CC$ such that $V(a,b)=T(a\circ b).$ If we take $a\in A$ such that $\pi_1(a)=I_{n_1}$, then for all $b,c\in A$ we  have 
		$$n_1 tr(\pi_1( b\circ c))=V(a,b\circ c)=T( I_{n_1}\circ(b\circ c))=2T(b\circ c)=2tr(\pi_1(b)) tr(\pi_1(c)).$$
		Take $b\in A$ such that $\pi_1(b)_{11}=1,\pi_1(b)_{22}=-1$ and let the rest of the entries of $\pi_1(b)$ be zero. Set $c=b.$ Then we have 
		$$4n_1= n_1tr(\pi_1( b\circ b))=2tr(\pi_1(b)) tr(\pi_1(b))=0,$$ a contradiction. Thus, $A$ is not symmetrically square-zero determined.
		
	\end{exm}
	
	\begin{rem}
		It is easy to realize,  in view of \cref{nonexample}, that any algebra that has a direct summand of the form $M_n(\CC)$ with $n>1$ is not symmetrically square-zero determined.
	\end{rem}
	
	Let $A$ be an associative algebra and $z\in A$. We say that $A$ is \textit{determined by Jordan products at $z$} (resp., \textit{symmetrically determined by Jordan products at $z$}) if for every vector space $X$ and bilinear (resp., symmetric bilinear) map $V:A\times A\to X$ satisfying 
	\begin{align}\label{a-circ-b=z=>V(a_b)=const}
		a\circ b=a'\circ b'=z \Rightarrow V(a,b)=V(a',b')
	\end{align}
	there exists a linear map $T:A\to X$ such that $V(a,b)=T(a\circ b)$, for all $a,b\in A.$ 
	
	We say that $A$ is \textit{zero Jordan product determined} or \textit{zJpd} (resp., \textit{symmetrically  zJpd}) if $A$ is (symmetrically) determined by Jordan products at $z=0.$ It is easy to see that every (symmetrically) square-zero determined associative algebra is (symmetrically) zJpd.

	If $A$ is a Banach algebra, then in the above definitions we assume $X$ to be a Banach space, $V$ a bounded bilinear map $A\times A\to X$ and $T$ a bounded linear map $A\to X$ in order to define \textit{(symmetrically) square-zero determined Banach algebras, zJpd determined Banach algebras or Banach algebras determined by Jordan products at $z$.} 
	
	
	\section{Square-zero determined von Neumann algebras and Jordan product property}\label{sec-sq-zero}
	
	For an algebra $A$ we denote by $Z(A)$ the center of $A.$ A C$^*$-algebra $A$ is said to be a \textit{von Neumann algebra} if it is dual to a Banach space (called \textit{predual} of $A$). It is known that each von Neumann algebra has a unique (up to isometry) predual \cite[Chapter III, Corollary 3.9]{Takesaki}. A von Neumann algebra is a \textit{factor} if $Z(A)= \CC 1,$ which by \cite[Proposition 1.10.5]{Sakai_Book} is equivalent to being simple as a von Neumann algebra.
	
	We denote by $P(A)$ the set of projections of the C$^*$-algebra $A.$ If $p,q\in P(A)$ we define $p\leq q$ iff $pq=p.$ An element $u\in A$ is said to be a \textit{partial isometry} if $uu^*$ is a projection (equivalently, if $u^*u$ is a projection). 
	Two projections $p,q$ in a C$^*$-algebra $A$ are said to be (Murray-von Neumann) \emph{equivalent}, $p\sim q$, if $p = uu^*$ and $q= u^*u$ for some partial isometry $u \in A$. Following \cite{Kad_Ring_2} if there exists $q_1\leq q$ such that $p\sim q_1$ then we write $p \precsim  q.$ We write $p\prec q$ to denote $p \precsim q$ and $p\nsim q.$ A projection $p$ is \textit{finite} if $p\sim q\leq p$ implies $q=p.$ Otherwise, $p$ is said to be \textit{infinite}. An infinite projection is \textit{properly infinite} if for $z\in P(Z(A))$ we have $zp \mbox{ finite }\impl zp=0.$
	
	A C$^*$-algebra is said to be \textit{finite}, \textit{infinite} or \textit{properly infinite} according to the respective property of its identity element.

	\begin{lem}\label{lem finite-infinite}
		Let $A$ be an infinite von Neumann algebra. 
		\begin{enumerate} 
			\item\label{fin+fin} If $p$ and $q$ are finite projections and $pq=0$, then $p+q$ is a finite projection.
			\item\label{1-fin} If $p$ is a finite projection, then $1-p$ is an infinite projection.
			\item\label{fin+inf} If $p$ is a finite projection, $q$ is an infinite projection and $pq=0$, then $p+q$ is an infinite projection.
			\item\label{factor prop inf} If $A$ is a factor, then every infinite projection in $A$ is properly infinite.
		\end{enumerate}
	\end{lem}
	\begin{proof}
		\cref{fin+fin} is a direct consequence of \cite[\S V, Theorem 1.37]{Takesaki}   since $p \vee q=p+q.$ 
		
		\cref{1-fin} If $1-p$ is finite, then $1=p+(1-p)$ is finite by \cref{fin+fin}, a contradiction since $A$ is assumed to be infinite.
		
		\cref{fin+inf} Suppose that $e=p+q$ is finite. Then $p\leq e$ implies that $p$ is finite in $eAe$ by \cite[Proposition 2.4]{Cuntz77} and hence $p$ is finite in $A$ by \cite[\S 15, Proposition 1]{Berberian_Book}, a contradiction.
		
		\cref{factor prop inf} Since $1$ is the only central projection, then $1$ is clearly properly infinite. Now, if $p\in P(A)$ is an infinite projection, then $1\cdot p$ is infinite, thus $p$ is properly infinite.
	\end{proof}
	For an algebra $A$, we denote by $N_2(A)$ the set of square-zero elements of $A$. The following property is inspired by the proof of \cite[Theorem 2.1]{Semrl-square-zero}:
	
	\begin{lem}\label{prop orth sqzero spans} Let $A$ be an algebra and $a=\sum_{i=1}^r a_i,b=\sum_{j=1}^s b_j$, where $\{a_i\}_{i=1}^r$, $\{b_j\}_{j=1}^s\sst N_2(A)$ and $a_ib_j=b_ja_i=0,\forall i,j$. Let $X$ be a vector space and $V:A\times A\to X$ be a symmetric bilinear map that has square-zero property. Then $V(a,b)=0.$
	\end{lem}
	\begin{proof}
		It is clear that $(a_i+b_j)^2=0,\forall i,j.$ We have
		$$0=V(a_i+b_j,a_i+b_j)=V(a_i,a_i)+V(b_j,b_j)+2V(a_i,b_j)=2V(a_i,b_j) $$ which yields $V(a_i,b_j)=0.$ By additivity we have $V(a,b)=0.$
	\end{proof}
	
	\begin{lem}\label{lem separate orth sqzero}
		Let $A$ be a von Neumann algebra, $X$ a vector space and $V:A\times A$ a symmetric bilinear map that has the square-zero property. If $p,q\in P(A)$ are properly infinite and $pq=0$ then $V(a,b)=0$ for all $a \in pAp$ and $b\in qAq.$
		
	\end{lem}
	\begin{proof}
		Fix $a\in pAp$ and $b\in qAq.$ Since both $pAp$ and $qAq$ are properly infinite von Neumann algebras, then by \cite[Theorem 5]{Pearcy-Topping-SmallSums} there exist $a_1,\dots,a_5\in N_2(pAp)\subset N_2(A)$ and $b_1,\dots,b_5 \in N_2( qAq)\subset N_2(A)$ such that $a=a_1+\dots +a_5$ and  $b=b_1+\dots +b_5.$ It is clear that $a_i b_j =b_ja_i=0,\forall i,j.$ The result now follows from \cref{prop orth sqzero spans}.
	\end{proof}
	Let $A$ be a C$^*$-algebra and let $a,b\in A.$  We say that $a$ and $b$ are \textit{orthogonal} (resp., \textit{have zero Jordan product}) if $ab^*=b^*a=0$ (resp., $a\circ b=0$). Set $A_{sa}=\{a\in A:a^*=a\}.$ It is clear that if $a,b\in A_{sa}$ are orthogonal, then $a\circ b=0.$ However, the converse is not true. In case $a\geq 0$ and $b\in A_{sa}$ by \cite[Lemma 4.1]{OP_REV} we have that $a$ is orthogonal to $b$ if and only if $a\circ b=0$ if and only if $ab=0.$
	
	Let $A$ be an algebra and $a\in A$. We denote by $L(\{a\})$ (resp, $R(\{a\})$) the \textit{left} (resp., \textit{right}) \textit{annihilator} of $a.$ Now let $A$ be a von Neumann algebra and $a\in A_{sa}.$ The \textit{support projection} of $A,$ denoted $s(a),$ is the smallest among all projections $p\in P(A)$ such that $ap=pa=a$
	(see \cite[Definition 1.10.3 and Theorem 1.12.1]{Sakai_Book}). Moreover, we also have $R(\{a\})=(1-s(a))A$ and $L(\{a\})=A(1-s(a)).$
	
	The next result shows that every simple infinite von Neumann algebra is a symmetrically square-zero determined Banach algebra.
	
	\begin{thrm}\label{thm BH square-zero determined} Let $X$ be a vector space, $A$ a simple infinite von Neumann algebra and let $V:A \times A\to X$ be a symmetric bilinear map that has square-zero property. Then
		\begin{align}\label{ab=0=>V(a_b)=0-for-symm-a_b}
			\forall a,b\in A_{sa}:\ ab=0\impl V(a,b)=0.
		\end{align}
		Moreover, $A$ is a symmetrically square-zero determined Banach algebra.   
	\end{thrm}
	
	\begin{proof}

		We first prove \cref{ab=0=>V(a_b)=0-for-symm-a_b}.
		
		Recall from \cref{lem finite-infinite}\cref{factor prop inf} that every infinite projection in $A$ is properly infinite.
		
		Fix $a,b\in A_{sa}$  with $ab=0$ and let $s(a),s(b)$ be their respective support projections. Observe that since $ab=ba=0$, then $s(a)s(b)=s(b)s(a)=0.$ We shall split the proof into several cases. 
		
		\textit{Case 1.} $s(a)$ and $s(b)$ are infinite. Then $s(a)$ and $s(b)$ are properly infinite and since $a\in s(a)As(a),b\in s(b)As(b)$, we can apply \cref{lem separate orth sqzero}.
		
		\textit{Case 2.} $s(a)$ and $s(b)$ are finite. Then $p'=1-s(a)-s(b)$ is infinite (and hence properly infinite) and $p',s(a),s(b)$ are pairwise orthogonal. There exists $e\leq p'$ such that $e\sim p'-e\sim p'$ (see \cite[Proposition 4.12]{Stratila_Book}). Then $e$ and $p'-e$ are infinite by \cite[Proposition 6.3.2]{Kad_Ring_2}. Set $p=s(a)+e$ and $q=s(b)+(p'-e).$ Then $p$ and $q$ are infinite (otherwise $e$ and $p'-e$ would be finite), $a\in pAp,b\in qAq$ and $pq=0.$ We can again apply \cref{lem separate orth sqzero}
		
		\textit{Case 3.} $s(a)$ is finite and $s(b)$ is infinite.
		
		\textit{Case 3.1.} $1-s(b)$ is infinite. Then $s(a)\leq 1-s(b)$ and we can take $p=1-s(b)$ and $q=s(b)$ and apply \cref{lem separate orth sqzero}.
		
		\textit{Case 3.2.}  $1-s(b)$ is finite.
		
		\textit{Case 3.2.1.} 
		Suppose that $b$ is a projection. Then $1-b$ is finite.  Again, it is clear that $s(a)\leq 1-b.$ We can find by \cite[Proposition 4.12]{Stratila_Book}  a properly infinite projection $e$ such that $e\sim b-e \sim b$ and $b=e+(b-e).$ By \cite[Proposition 6.3.2]{Kad_Ring_2} $b-e$ is also properly infinite. Set $p_1=(1-b)+e.$ Then $p_1$ is infinite by \cref{lem finite-infinite}\cref{fin+inf}, $a\in p_1 A p_1$ and $p_1\perp b-e,$ thus $V(a,b-e)=0$ by Case 1. Similarly, setting $p_2=(1-b)+(b-e)$ we have $V(a,e)=0.$ Thus $V(a,b)=V(a,e)+V(a,b-e)=0.$ 
		
		\textit{Case 3.2.2.} $b\in A_{sa}$.
		Now, since $b\in s(b)As(b)$ and the latter is a properly infinite von Neumann algebra, then by \cite{Pearcy-Topping-SmallSums} there exist $q_1,\dots,q_8 \in P(A)$ and $\lambda_1,\dots,\lambda_8 \in \mathbb{R}$ such that $b=\sum_{\lambda=1}^8\lambda_i q_i.$ For each $i\in \{1,\dots,8\},$ if $q_i$ is finite, then $V(a,q_i)=0$ by Case 2, otherwise $q_i$ is infinite, in which case we can apply Case 3.2.1 to show that $V(a,q_i)=0.$ Finally, we have $V(a,b)=\sum_i \lambda_i V(a,q_i)=0.$

		\textit{Case 4.} $s(a)$ is infinite and $s(b)$ is finite. This case follows from Case 3 since $V$ is symmetric. 
		
		The proof of \cref{ab=0=>V(a_b)=0-for-symm-a_b}  is complete.\smallskip
		
		
		
		
		
		Now let us assume that $X$ is a Banach space and $V$ is bilinear and continuous. Fix $\psi\in X^*$ and set $V_{\psi}=\psi \circ V.$ By the first part of the proof $V_{\psi}$ is orthogonal on $A_{sa}$ in the sense of \cite[Proposition 3.2]{Peralta-Jordan-weakamen} and hence there exists a unique $\vf\in A^*$ such that $V_{\psi}(a,b)=\vf(a\circ b), \; a,b\in A.$ In particular, 
		$$ a\circ b=0 \Rightarrow \psi( V(a,b))=\vf(a\circ b)=0.$$ 
		Applying the Hahn-Banach theorem, one sees that $V$ satisfies \cref{a-circ-b=z=>V(a_b)=const} with $z=0$ for all $a,b\in A$. Then, since $A$ is zJpd by \cite[Proposition 4.9 and Corollary 6.25]{Bresar-zpd}, the result follows.
	\end{proof}
	

	\begin{cor}\label{cor direct sums sqz}
		Let $A_1,\dots,A_n$ be infinite simple von Neumann algebras. Then $A_1\oplus \dots \oplus A_n$ is a symmetrically square-zero determined Banach algebra.
	\end{cor}
	\begin{proof}
		Let us write $A=A_1\oplus \dots \oplus A_n.$ Let $X$ be a Banach space and $V:A\times A\to X$ a bounded symmetric bilinear map that has the square-zero property. We are going to show that $V$ has Jordan product property at zero. Then the result will follow from \cite[Proposition 4.9 and Corollary 6.25]{Bresar-zpd}. We shall split the proof into two cases.  Let $a,b\in A$ with $a\circ b=0$.
		
		\textit{Case 1.} $a,b\in A_i$. Then $V(a,b)=V_i(a,b)=0$, where $V_i:=V|_{A_i \times A_i}$, which has the Jordan product property at zero by \cref{thm BH square-zero determined}.
		
		\textit{Case 2.} $a\in A_i,b\in A_j$ with $i\neq j$. If $a^2=b^2=0$, then $(a+b)^2=0$ and hence $V(a+b,a+b)=0.$ It follows that $V(a,b)=0.$ Otherwise by \cite[Theorem 5]{Pearcy-Topping-SmallSums} we can find $a_k\in A_i,b_l\in A_j,k,l=1,\dots, 5$, such that $a_k^2=b_l^2=0$ for all $k,l$ and $a=a_1+\dots+a_5$ and $b=b_1+\dots+b_5$. Then $V(a_k,b_l)=0,\forall k,j.$ Thus $V(a,b)=0$ by bilinearity.
		
		The general case $a,b\in A$ follows from Cases 1 and 2 by the bilinearity of $V$.
	\end{proof}

	Given a family $(A_{\lambda})$ of Banach algebras, let $\ell_{\infty}(\Lambda,A_{\lb})$ be the set of all $(a_{\lb})\in \prod_{\lb} A_{\lb}$ such that $\|(a_\lb)\|:=\sup \{ \|a_\lb\|\}<\infty.$ Then $\ell_{\infty}(\Lambda,A_{\lb})$ is a Banach algebra. Moreover, if each $A_{\lambda}$ is a C$^*$-algebra, then $\ell_{\infty}(\Lambda,A_{\lb})$ is also a C$^*$-algebra. The algebra $\ell_{\infty}(\Lambda,A_{\lb})$ contains the C$^*$-subalgebra $C_0(\Lambda,A_{\lb})$ formed by $(a_{\lb})\in \prod_{\lb} A_{\lb}$ with $\lim \|a_\lb\|=0$. Furthermore, if $\bigoplus A_{\lambda}=\{a\in \prod_{\lb} A_{\lb}:|\{\lambda\in \Lambda: a_{\lb}\neq 0\}|<\infty\}$ then $C_0(\Lambda,A_{\lb})$ is the norm-closure of $\bigoplus A_{\lambda}$ inside $\ell_{\infty}(\Lambda,A_{\lb}).$

	\begin{cor} \label{cor sums B(H) determined normal}
		Let $(A_{\lambda})_{\lambda \in \Lambda}$ be a family of simple infinite von Neumann algebras. Then $C_0(\Lambda,A_{\lb})$ is a symmetrically square-zero determined Banach algebra.
	\end{cor}
	\begin{proof} 
		Let $X$ be a vector space and $V:C_0(\Lambda,A_{\lb})\times C_0(\Lambda,A_{\lb}) \to X$ be a symmetric bounded bilinear map. As in previous results, it is enough to show that $V$ has the Jordan product property at zero.
		
		Take $a,b\in \bigoplus A_{\lambda}$ with $a\circ b=0.$ Set $J=\{ \lambda\in \Lambda: a_\lb\neq 0 \mbox{ or } b_\lb\neq 0\}.$ Clearly $|J|<\infty.$ Define $1_J=\sum_{\lambda\in J} 1_{A_{\lambda}}.$ It is clear that $\|1_J\|=1.$ Then $1_J,a,b\in \bigoplus_{\lambda \in J} A_{\lambda}$ and by  \cref{cor direct sums sqz} we have 
		\begin{align*}
			V(a,b)=T(a\circ b)=\frac 12 T(1_J \circ (a\circ b))=\frac 12 V(1_J,a\circ b) ,\\
			\mbox{ and }  \|V(a,b)\| \leq \frac 12\|V\| \|1_J\| \|a\circ b\|=\frac 12\|V\| \|a\circ b\|
		\end{align*}  by continuity of $V.$   Now take $a,b\in C_0(\Lambda,A_{\lb})$ with $a\circ b=0.$ There exist $(a_n),(b_n)\in \bigoplus A_{\lambda}$ such that $\lim_n a_n=a$ and $\lim_n b_n=b.$ We have $$\| V(a,b)\|=\lim_n\|V(a_n,b_n)\|\leq \frac 12\| V\| \lim_n \| a_n\circ b_n\|=\frac 12\|V\| \|a\circ b\|=0, $$ hence $V$ has Jordan product property at zero. 
	\end{proof}

	
	
	

	\section{Transferring square-zero property to Jordan product property}\label{sec-Jordan-property}
	
	In this section, we deal with bilinear maps that have the Jordan product property at a fixed point. 
	Our first two lemmas are inspired by \cite[Lemma 1.3]{Chen_jordan_der} and \cite[Theorem 2]{Catalano-Hsu-Kapalko19}, respectively. Similar ideas can also be found in the more recent work \cite{Gen_jord_SU}. For $\alpha\in \CC,r>0$ we denote by $D(\alpha,r)$ (resp., $\overline{D}(\alpha,r)$) the open (resp., closed) disc of center $\alpha$ and radius $r.$

	\begin{lem}\label{lem hol funct calculus}
		Let $A$ be a unital Banach algebra and $a\in A.$ For every $\lambda >\|a\|$ there exists $b_{\lambda}\in A$ such that $b_{\lambda}^2=1-\frac{1}{\lambda}a$. 
	\end{lem}
	\begin{proof}
		It is well known that $sp(a)\subseteq \overline{D}(0,\|a\|)$ and hence $sp( \frac{1}{\lambda}a) \subseteq \overline{D}(0,\frac{\|a\|}{\lambda })\subset D(0,1).$
		Let $f(z)$ be the principal branch of $z^{\frac{1}{2}}.$ The function $g(z)=f(1-z)$ is holomorphic in the disc $D(0,1)$ and $g(z)^2=1-z.$ Let us define, via holomorphic functional calculus (see, for instance, \cite[Theorem 3.3.5]{Kad_Ring_1}) $ b_{\lambda}=g( \frac{1}{\lambda}a).$ Then $b_{\lambda}^2=1-\frac{1}{\lambda}a.$ 
	\end{proof}
	
	A C$^*$-algebra $A$ is said to be a \textit{Rickart C$^*$-algebra} if for every $x\in A$ there exists $p\in P(A)$ such that $R(\{x\})=pA$. If $A$ is a Rickart C$^*$-algebra, then for every $x$ there also exists $q\in P(A)$ such that $L(\{x\})=qA$. By \cite[\S 3, Proposition 3]{Berberian_Book} for every $x\in A$ there exists a projection denoted by $RP(x)$ (respectively, $LP(x)$) such that $xRP(x)=x$, and we have $xy=0$ iff $RP(x)y=0$ (respectively, $LP(x)x=x$, and we have $yx=0$ iff $yLP(x)=0$). It is worth noticing that every Rickart C$^*$-algebra is unital \cite[\S 3, Proposition 2]{Berberian_Book}.

	\begin{lem}\label{lem transfer pp normal to sq-zero}
		Let $A$ be a unital C$^*$-algebra, $X$ be a vector space and let $V:A\times A\to  X$ be a symmetric bilinear map that has Jordan product property at a fixed element $z\in A.$  
		\begin{enumerate}
			\item \label{a^2-2b^2=z} If $2a^2-2b^2=z$ then $2V(a,a)-2V(b,b)=V(1,z).$
			\item \label{a^2=b^2=1} If $a^2=b^2=1$ then $V(a,a)=V(b,b).$
			\item\label{transfer to square-zerp} If $A$ is a Rickart C$^*$-algebra then $V$ has square-zero property.
		\end{enumerate}  
	\end{lem}
	\begin{proof}
		\cref{a^2-2b^2=z} The property follows from $(a+b)\circ (a-b)=1\circ \frac z2=z,$ Jordan product property and bilinearity.
		
		\cref{a^2=b^2=1} Take $n$ such that $\|z\|<2n^2.$ By \cref{lem hol funct calculus} there exists $x\in A$ such that $x^2=1+\frac{z}{2n^2}.$ Define $y=n x.$ Since $y^2=n^2x^2=n^2 1+ \frac{z}{2}$ we have $$ 2y^2-2(na)^2=2y^2-2(nb)^2=z.$$ By \cref{a^2-2b^2=z} we have $$2V(y,y)-2V(na,na)=V(1,z)=2V(y,y)-2V(nb,nb)$$ and hence $V(a,a)=V(b,b).$
		
		\cref{transfer to square-zerp} Let $a\in A$ be a square-zero element and $p=LP(a).$ Then $pa=a$ and since $aa=0$ we have $ap=0.$ Set $u=2p-1.$ It is easy to see that $u^2=1$ and $a\circ u=0.$ Thus $1=u^2=(u+a)^2=(u-a)^2.$ By \cref{a^2=b^2=1} we have
		\begin{align*}
			V(u,u)=V(u,u)+2V(a,u)+V(a,a)=V(u,u)-2V(a,u)+V(a,a).
		\end{align*} 
		whence $V(a,a)=0.$
	\end{proof}
	
	\begin{rem}\label{rem i-ii implies iii}
		It is not hard to see that if $V:A\times A\to X$ is a symmetric bilinear map that satisfies \cref{lem transfer pp normal to sq-zero} \cref{a^2=b^2=1} and $A$ is a Rickart C$^*$-algebra, then $V$ also satisfies \cref{lem transfer pp normal to sq-zero}\cref{transfer to square-zerp}.
	\end{rem}
	
	As a consequence of \cref{lem transfer pp normal to sq-zero}\cref{transfer to square-zerp}, a symmetrically square-zero determined Rickart C$^*$-algebra is symmetrically determined by Jordan products at any point. Since every von Neumann algebra is a Rickart C$^*$-algebra (see \cite[\S 1.10]{Sakai_Book}), we have as a consequence of \cref{cor direct sums sqz}:
	
	\begin{thrm}\label{cor sums A simple determined any}
		Finite direct sums of simple infinite von Neumann algebras are symmetrically determined, as Banach algebras, by Jordan products at any fixed element.
	\end{thrm}

	\section{Applications to Jordan homomorphisms at a point and derivable maps}\label{sec-appl}

	The following result can be seen as an answer to the problem posed by {\v{S}}emrl in \cite{Semrl-square-zero} and generalizes \cite[Theorem 2.3]{Bai_Hou_Annih_Polyn} for bounded linear maps between more general von Neumann algebras. 
	
	\begin{thrm}\label{thrm vf pp at square zero}
		Let $A_1,\dots,A_n$ be infinite simple von Neumann algebras, $B$ a Banach algebra and $A=A_1\oplus \dots \oplus A_n.$ Let $\vf:A\to B$ be a bounded linear map. If $\vf$ preserves square-zero elements or is a Jordan homomorphism at $z\in A,$ then $\vf$ preserves zero Jordan products.
		
		Moreover, if $B$ has a bounded approximate identity and $\vf$ is surjective, then $\vf(1)$ is central and invertible, and there exists a surjective Jordan homomorphism $\psi:A\to B$ such that $\vf(-)=\vf(1)\psi(-).$  
	\end{thrm}
	
	\begin{proof}
		If $\vf$ preserves square-zero elements, then
		the bilinear map $V(a,b)=\vf(a)\circ \vf(b)$ is bounded, symmetric and has square-zero property. 
		By \cref{cor direct sums sqz}  there exists a bounded linear map $T:A\to B$ such that $\vf(a)\circ \vf(b)=V(a,b)=T(a\circ b)$ for all $a,b\in A.$ As a consequence $\vf$ preserves zero Jordan products. In case $\vf$ is a Jordan homomorphism at $z$, then the same conclusion can be obtained by \cref{cor sums A simple determined any}.
		
		Finally, the second statement follows from \cite[Theorem 3.3]{Bresar-Godoy-Villena22} and its proof. 
	\end{proof}

	\begin{rem}
		If $n=1$ in \cref{thrm vf pp at square zero}, then by \cite[Corollary 5.4]{jordan-hom-rev} and the simplicity of $A,$ the map $\vf$ is a multiple of either a homomorphism or an antihomomorphism.
	\end{rem}
	Let $A$ be a C$^*$-algebra. The \textit{multiplier algebra} of $A$ is defined as 
	$$
	M(A)=\{a\in A^{**}:aA,Aa\subseteq A\}.
	$$
	It is known that $M(A)$ is a unital C$^*$-algebra and $A$ is an ideal in $M(A).$ If $A$ is unital, then $M(A)=A$ (see \cite{Multipliers-Ak-Ped}). 
	
	We can generalize \cref{thrm vf pp at square zero} to $c_0$-sums of infinite simple von Neumann algebras, when $\vf$ is a square-zero preserver, but this needs to be treated independently since the C$^*$-algebra $A=C_0(\Lambda,A_{\lb})$ is not unital even if all the $A_{\lambda}$'s are. 
	
	Recall that, given a bounded linear map $\vf:X\to Y$ between Banach spaces, it can be canonically extended to its second adjoint $\vf^{**}:X^{**}\to Y^{**}$, also known as the bidualization of $\phi$. Now, if $B$ is Banach algebra, a linear map $W:B\to B$ is said to be a \textit{centralizer} if $W(ab)=aW(b)=W(a)b$ for all $a,b\in B.$
	
	\begin{cor}\label{Jordan-homo-at-point-c_0-sum}
		Let $A=C_0(\Lambda,A_{\lb})$, where $(A_{\lambda})_{\lambda \in \Lambda}$ is a family of simple infinite von Neumann algebras and $B$ a Banach algebra. Let $\vf:A\to B$ be a bounded linear map. If $\vf$ preserves square-zero elements, then $\vf$ preserves zero Jordan products. 
		
		Moreover, if $B$ has a bounded approximate identity and $\vf$ is surjective, then the linear map defined by $ W(u)=\vf^{**}(1_{A^{**}})u$ for all $u\in B$ is an invertible centralizer of $B$ and there exists a surjective Jordan homomorphism $\psi:A\to B$ such that $\vf=W\circ \psi.$  
	\end{cor}
	
	\begin{proof}
		We can follow the part of the proof of \cref{thrm vf pp at square zero} that deals with square-zero preservers, we just need to replace \cref{cor direct sums sqz} with \cref{cor sums B(H) determined normal} in order to show that $\vf$ preserves zero Jordan products. Finally, the result follows again by \cite[Theorem 3.3]{Bresar-Godoy-Villena22} and its proof.
	\end{proof}
	
	We now turn our attention to Jordan derivations at a point. A Banach $A$-bimodule $X$ is said to be \textit{essential} if it coincides with the closed linear span of $AXA.$ 
	

	
	\begin{thrm}\label{thrm derivable maps}
		Let $A=A_1\oplus \dots \oplus A_n,$ where $A_1,\dots,A_n$ are infinite simple von Neumann algebras, and let $X$ be an essential Banach $A$-bimodule.
		If a bounded linear map $\delta:A\to X$ is a square-zero derivation or Jordan derivable at a fixed element $z\in A$, then $x\delta(1)=\delta(1)x$, for all $x\in X$, and $\delta(1)z=0$. Moreover, there exists a derivation $d:A\to X$  such that 
		\begin{align}\label{dl(a)=d(a)+dl(1)a}
			\delta(a)=d(a)+\delta(1)a   
		\end{align}
		for all $a\in A$.     
	\end{thrm}
	
	\begin{proof}
		Define the bounded bilinear map 
		\begin{align*}
			V(a,b)=\delta(a\circ b)-\delta(a)\circ b-a\circ \delta(b).
		\end{align*}
		If $\delta$ is a square-zero derivation, then $V$ has the square-zero property and by \cref{cor direct sums sqz} there exists a bounded linear map $T:A\to B$ such that $V(a,b)=T(a\circ b)$ for all $a,b\in A.$ As a consequence, $\delta$ is Jordan derivable at zero. In case $\delta$ is Jordan derivable at some $z\in A$  by applying \cref{cor sums A simple determined any} we again see that $\delta$ is Jordan derivable at zero. By \cite[Theorem 8.6]{Bresar-zpd} and the comments after it $\delta(1)x=x\delta(1),$ for all $x\in X$, and there exists a derivation $d:A\to X$  such that \cref{dl(a)=d(a)+dl(1)a} holds for all $a\in A$. Finally, if $a\circ b=z$, using \cref{dl(a)=d(a)+dl(1)a}, we have
		\begin{align*}
			d(a)\circ b+a\circ d(b)+\delta(1)z=d(a\circ b)+\dl(1)(a\circ b)=\delta(a\circ b)=\delta(a)\circ b+a\circ \delta(b)\\
			=d(a)\circ b+a\circ d(b)+2\delta(1)z,
		\end{align*} whence $\delta(1)z=0.$
	\end{proof}

	Since every derivation on a von Neumann algebra is inner (see \cite[Theorem 4.1.6]{Sakai_Book}) we have:

	\begin{cor}\label{cor derivable maps X=A}
		Let $A=A_1\oplus \dots \oplus A_n,$ where $A_1,\dots,A_n$ are infinite simple von Neumann algebras.
		If a bounded linear map $\delta:A\to A$ is a square-zero derivation or Jordan derivable at a fixed element $z\in A$, then $\delta(1)\in Z(A),$ $\delta(1)z=0$ and there exists $a\in A$ such that $\delta(a)=[a,x]+\delta(1)a$ for all $a\in A.$     
	\end{cor}
	
	\begin{rem}\label{rem jordan dar at z on factor is derivation }
		If in \cref{cor derivable maps X=A} we have $n=1$ then, since $A$ is a factor, it follows that $\dl(1)\in\mathbb{C}1$. Thus, if $z\neq 0$, then $\delta(1)=0$ and hence $\delta$ is a derivation. Consequently, in this context, a Jordan derivable map at zero is also Jordan derivable at some $z\ne 0$ if and only if it is a derivation.
	\end{rem}
	

	\section*{Acknowledgements}
	Jorge J. Garcés was partially supported by grant PID2021-122126NB-C31 funded by MCIN/AEI/10.13039/501100011033 and by ERDF/EU and Junta de Andalucía grant FQM375. Part of this work was completed during a visit of Jorge J. Garces to the Universidade do Porto, which he thanks for the hospitality. Mykola Khrypchenko was partially supported by CMUP, member of LASI, which is financed by national funds through FCT --- Fundação para a Ciência e a Tecnologia, I.P., under the project with reference UIDB/00144/2020.
	
	The authors thank the reviewer for the careful reading of the manuscript and for the valuable comments that resulted in significant changes in \cref{sec-appl}.
	
	\bibliography{bibl2}{}
	\bibliographystyle{acm}
	
\end{document}